\documentclass[12pt,reqno]{amsart}
\usepackage{amscd,amsmath,amsthm,amssymb}
\usepackage{color}
\usepackage{stmaryrd}
\usepackage{graphicx}  
\usepackage{leftidx}
\usepackage{tikz}
\usepackage{tikz-cd}
\usepackage{cleveref}
\usepackage{bm}

\def\NZQ{\mathbb}               

\def\ZZ{{\NZQ Z}}

\def\frk{\mathfrak}               

\def\mm{{\frk m}}

\def\G{{\mathcal G}}

\def\pd{\textup{proj}\phantom{.}\!\textup{dim}}

\def\opn#1#2{\def#1{\operatorname{#2}}} 
\opn\chara{char} \opn\length{\ell} \opn\pd{pd} \opn\rk{rk}
\opn\projdim{proj\,dim} \opn\injdim{inj\,dim} \opn\rank{rank}
\opn\depth{depth} \opn\grade{grade} 
\opn\Rees{Rees} \opn\height{height}
\opn\embdim{emb\,dim} \opn\codim{codim}

\opn\Tr{Tr} \opn\bigrank{big\,rank}
\opn\superheight{superheight}\opn\lcm{lcm}
\opn\trdeg{tr\,deg}
	\opn\reg{reg} \opn\lreg{lreg} \opn\ini{in} \opn\lpd{lpd}
	\opn\size{size} \opn\sdepth{sdepth}
	\opn\link{link}\opn\fdepth{fdepth}\opn\lex{lex}
	\opn\tr{tr}
	\opn\type{type}
	\opn\gap{gap}
	\opn\diam{diam}
	\opn\Mod{Mod}
	\opn\revlex{revlex}
	\opn\div{div} \opn\Div{Div} \opn\cl{cl} \opn\Cl{Cl}
	\opn\Spec{Spec} \opn\Supp{Supp} \opn\supp{supp} \opn\Sing{Sing}
	\opn\Ass{Ass} \opn\Min{Min}\opn\Mon{Mon}
	\opn\Ann{Ann} \opn\Rad{Rad} \opn\Soc{Soc}
	\opn\Im{Im} \opn\Ker{Ker} \opn\Coker{Coker} \opn\Am{Am}
	\opn\Hom{Hom} \opn\Tor{Tor} \opn\Ext{Ext} \opn\End{End}
	\opn\Aut{Aut} \opn\id{id}
	
	\opn\nat{nat}
	\opn\pff{pf}
	\opn\Pf{Pf} \opn\GL{GL} \opn\SL{SL} \opn\mod{mod} \opn\ord{ord}
	\opn\Gin{gin} \opn\Hilb{Hilb}\opn\sort{sort}
	\opn\PF{PF}\opn\Ap{Ap}
	\opn\dist{dist}
	\opn\aff{aff}
	\opn\relint{relint} \opn\st{st}
	\opn\lk{lk} \opn\cn{cn} \opn\core{core} \opn\vol{vol}  \opn\inp{inp} \opn\nilpot{nilpot}
	\opn\link{link} \opn\star{star}\opn\lex{lex}\opn\set{set}
	\opn\width{wd}
	\opn\Fr{F}
	\opn\QF{QF}
	\opn\G{G}
	\opn\type{type}\opn\res{res}
	\opn\conv{conv}
	\opn\sr{sr}
	\opn\gr{gr}
	\def\Rees{{\mathcal R}}
	\def\pot#1#2{#1[\kern-0.28ex[#2]\kern-0.28ex]}

	\opn\dirlim{\underrightarrow{\lim}}
	\opn\inivlim{\underleftarrow{\lim}}
	\let\sect=\cap
	\let\dirsum=\oplus
	
	\let\iso=\cong

	\let\Dirsum=\bigoplus
	
	\let\to=\rightarrow
	\let\To=\longrightarrow
	\def\Implies{\ifmmode\Longrightarrow \else
		\unskip${}\Longrightarrow{}$\ignorespaces\fi}
	\def\implies{\ifmmode\Rightarrow \else
		\unskip${}\Rightarrow{}$\ignorespaces\fi}
	\def\iff{\ifmmode\Longleftrightarrow \else
		\unskip${}\Longleftrightarrow{}$\ignorespaces\fi}

	\let\:=\colon
	\newtheorem{Theorem}{Theorem}[section]
	\newtheorem{Lemma}[Theorem]{Lemma}
	\newtheorem{Corollary}[Theorem]{Corollary}
	\newtheorem{Proposition}[Theorem]{Proposition}
	
	\newtheorem{Remarks}[Theorem]{Remarks}
	\newtheorem{Example}[Theorem]{Example}
	
	\newtheorem{Definition}[Theorem]{Definition}

	\newtheorem{Question}[Theorem]{Question}
	
	\let\epsilon\varepsilon
	\let\kappa=\varkappa
	\def\qed{\ifhmode\textqed\fi
		\ifmmode\ifinner\hfill\quad\qedsymbol\else\dispqed\fi\fi}
	\def\textqed{\unskip\nobreak\penalty50
		\hskip2em\hbox{}\nobreak\hfill\qedsymbol
		\parfillskip=0pt \finalhyphendemerits=0}
	\def\dispqed{\rlap{\qquad\qedsymbol}}
	
	\opn\dis{dis}
	\def\pnt{{\raise0.5mm\hbox{\large\bf.}}}
	
	\opn\Lex{Lex}
	\opn\Max{Max}
	\opn\Shad{Shad}
	\opn\astab{astab}

	\opn\Fitt{Fitt}
	
	\opn\v{v}
	
\newcommand{\Sasha}[1]{{\color{violet} \sf $\heartsuit$ [#1]}}

\begin{document}

\title{}
	\author{Oleksandra Gasanova, J\"urgen Herzog, Filip Jonsson Kling, Somayeh Moradi}

 \address{Oleksandra Gasanova, Faculty of Mathematics, University of Duisburg-Essen, 45117 Essen, Germany}
	\email{oleksandra.gasanova@uni-due.de}

	\address{J\"urgen Herzog, Faculty of Mathematics, University of Duisburg-Essen, 45117 Essen, Germany} \email{juergen.herzog@uni-essen.de}

 \address{Filip Jonsson Kling, Department of Mathematics, Stockholm University, 106 91 Stockholm, Sweden}
	\email{filip.jonsson.kling@math.su.se}
	
	\address{Somayeh Moradi, Department of Mathematics, Faculty of Science, Ilam University, P.O.Box 69315-516, Ilam, Iran}
	\email{so.moradi@ilam.ac.ir}
	
	\thanks{Oleksandra Gasanova is supported by the Wenner-Gren Foundation}
	\thanks{Somayeh Moradi is supported by the Alexander von Humboldt Foundation.
	}
	
	\subjclass[2020]{Primary 13A30; Secondary 05E40.}
	
	\keywords{Rees algebra, conductor, ideal of linear type}
	
\title{On the Rees algebra and the conductor of an ideal}
\maketitle
\textit{In memory of Jürgen Herzog, who sadly passed away during
the finalizing of this paper. We will miss his expertise, enthusiasm,
and dedication to our research. This paper is a testament to his impact on our work,
and we hope his memory will continue to inspire us.}

\begin{abstract}
For an ideal $I$ in a Noetherian ring $R$, we introduce and study its conductor as a tool to explore the Rees algebra of $I$. The conductor of $I$ is an ideal $C(I)\subset R$ obtained from the defining ideals of the Rees algebra and the symmetric algebra of $I$ by a colon operation. Using this concept we investigate when adding an element to an ideal preserves the property of being of linear type. In this regard, a generalization of a result by Valla in terms of the conductor ideal is presented.
When the conductor of a graded ideal in a polynomial ring is the graded maximal ideal, a criteria is given for when the Rees algebra and the symmetric algebra have the same Krull dimension.
Finally, noting the fact that the conductor of a monomial ideal is a monomial ideal, the conductor of some families of monomial ideals, namely bounded Veronese ideals and edge ideals of graphs, are determined.  
\end{abstract}


\section*{Introduction}


Let $I=(f_1,\ldots,f_n)$ be an ideal in a Noetherian ring $R$. The Rees algebra of $I$ is the graded $R$-algebra 
$
\mathcal{R}(I) = R[f_1t, \ldots, f_mt] \subset R[t]
$. A natural way to represent the Rees algebra $\mathcal{R}(I)$ is to consider the surjective  $R$-algebra homomorphism $\varphi : T \to \mathcal{R}(I)$ from the polynomial ring 
 $T = R[y_1, \ldots, y_m]$ to $\mathcal{R}(I)$  defined by $\varphi(y_j) = f_jt$ for all $j$.  Let $J=\Ker \varphi$ be the defining ideal of $\mathcal{R}(I)$. Finding the
 defining equations of $\mathcal{R}(I)$ is a challenging problem in general and has been an active area of research in commutative algebra and algebraic geometry for decades.
Some obvious elements in $J$ arise from the relations of $I$. Indeed, let 
\[
\Dirsum_{j=1}^sRg_j\stackrel{\gamma}{\To}\Dirsum_{i=1}^mRe_i\to I\to 0, \quad \text{$e_i\mapsto f_i$ for $i=1,\ldots,m$}
\]
be a free presentation of $I$, and let $\gamma(g_j)=\sum_{i=1}^ma_{ij}e_i$ for $j=1,\ldots,s$.  Then 
\[
\sum_{i=1}^m  a_{ij}y_i\in J \quad \text{for $j=1,\ldots,s$.}
\]
Relations of this type are called {\em linear relations}\index{linear relations}, and $\Rees(I)$ is called of {\em linear type}\index{linear type} if the linear relations are the {\em only} generating relations  of $\Rees(I)$. Complete intersection
ideals and $d$-sequence ideals are some classes of ideals of linear type ~\cite{Huneke}, ~\cite{Micali}. 

We denote by $L\subseteq J$ the ideal generated by the linear relations of $\Rees(I)$.
Let $S(I)$ be the symmetric algebra of $I$. 
There is a natural $R$-algebra epimorphism $\alpha: S(I)\to \mathcal{R}(I)$. Note that the kernel of this map is isomorphic to $J/L$. Therefore, the ideal $I$ is of linear type if and only if this map is an isomorphism. 

In this article we explore when the defining ideal of $\mathcal{R}(I)$ can be obtained from its linear relations by some colon operations. To this aim we consider the ideal $L:_RJ\subset R$ attached to $I$. We denote this ideal by $C(I)$. When $R$ is a domain, for each non-zero element $r\in C(I)$ one gets $J=L:r$, see \Cref{equalJ}. We investigate how the algebraic properties of $C(I)$ affects the Rees algebra $\mathcal{R}(I)$.

The paper proceeds as follows. In Section~1 we study some algebraic properties of $C(I)$. The variety of $C(I)$ is interpreted in terms of the linear type property of localizations of $I$, see~\Cref{variety}. From this description we deduce that $\sqrt{I}\subseteq \sqrt{C(I)}$.
When $I$ is a graded ideal in the polynomial ring $R$ over a field and  $C(I)$ is $\mm$-primary, a criterion for the equality $\dim \Rees(I)=\dim S(I)$ is given in \Cref{equalReesSym}. Here $\mm$ denotes the graded maximal ideal of $R$.  In \Cref{multigrading} it is shown that $C(I)$ inherits the $\ZZ^n$-graded structure from the ideal $I$. This fact is mainly useful when working with monomial ideals, as we see in Section~3.

Motivated by a result of Valla~\cite{Valla}, in Section~2 we ask for the conditions on an element $f\in R$ such that if $I$ is of linear type, then $(I,f)$ is of linear type as well. Note that $I$ is of linear type precisely when $C(I)=R$. So in order to answer the above question, it is natural to compare the ideals $C((I,f))$ and $(C(I),f)$. When $I$ is a monomial ideal, and $f$ is a monomial which is a non-zerodivisor modulo $I$, such comparison is given in \Cref{regularmodI}, from which it is deduced that $I$ is of linear type if and only if $(I,f)$ is of linear type. This result is not true in general if we replace a monomial ideal by a graded ideal. For an element $f$ which is a non-zerodivisor modulo all powers of a given ideal $I$, in \Cref{generalizationvalla} we show that $(C(I), f^k)\subseteq C((I,f))$ for some integer $k>0$. This implies the result proved by Valla, which is the particular case of the above inclusion stating that if $C(I)=R$, then $C((I,f))=R$.

In Section~3 we compute the conductor ideal $C(I)$ for bounded Veronese ideals and edge ideals of graphs, see~\Cref{veroneseConductor} and~\Cref{Graph_conductor}, respectively. The surprising fact is that the conductor of  bounded Veronese ideals are again  bounded Veronese ideals.  Using these characterizations, in \Cref{BVlinear} we determine when such ideals are of linear type.

Finally, Section~4 contains some questions and open problems.
 
\medskip
\section{Relations of $\Rees(I)$ and the conductor ideal of $I$}

Let $R$ be a Noetherian ring, $I$ an ideal in $R$, and let  $\alpha\: S(I)\to \Rees(I)$ be the canonical $R$-module homomorphism.  We let  $J$ and $L$ be the ideals in $T$, defined as before. Then $\Ker{\alpha}=J/L$.

We define the ideal 
$$C(I)=\{r\in R\:\;   J\subseteq L:_Tr\}$$
in $R$,
and call it the {\em conductor} of $I$.
Then
$$C(I)=L:_R J=(L:_T J)\cap R=\Ann_R(\Ker(\alpha)).$$ 

The conductor ideal of $I$ localizes. Indeed, we have

\begin{Lemma}
\label{cilocal}
 Let $S\subseteq R$ be a multiplicatively closed set. Then    \[ S^{-1}C(I)=C(S^{-1}I).
 \]
\end{Lemma}

\begin{proof}
We have $C(I)=\Ann_T(\Ker(\alpha))\sect R$. Since $\Ann_T(\Ker(\alpha))$ is a finitely generated $T$-module, localization commutes with taking the annihilator, and we get
\begin{eqnarray*}
S^{-1}C(I)&=& S^{-1}(\Ann_T(\Ker(\alpha))\sect R)= S^{-1}\Ann_T(\Ker(\alpha))\sect S^{-1}R\\
&=& \Ann_{S^{-1}T}(S^{-1}\Ker(\alpha))\sect S^{-1}R=\Ann_{S^{-1}T}(\Ker(S^{-1}\alpha))\sect S^{-1}R\\
&=&C(S^{-1}I).
\end{eqnarray*}
\end{proof}

As a first application we obtain

\begin{Proposition}\label{equalJ} 
Let $R$  be a domain and $I\subset R$ be a non-zero ideal. Then 
\begin{enumerate}
    \item[(a)] $C(I)\neq 0$.

    \item[(b)] $L:r=J$ for all $r\in C(I)\setminus \{0\}$. 
\end{enumerate} \end{Proposition}

\begin{proof}
(a) Suppose $C(I)=0$, and let $S=R\setminus\{0\}$. Then 
\[
0=S^{-1}C(I)=C(S^{-1}(I))=C(Q(R))=Q(R),
\]
where $Q(R)$ is the field of fractions of $R$, giving a contradiction. 

(b) Suppose $f\in L:r$ for some $f\in T$. Then $fr\in L\subseteq J$. Our assumption implies that $\Rees(I)$ is a domain. Therefore, $J$ is a prime ideal, and hence  $f\in J$ because  $r\not\in  J$. This implies that $L:r\subseteq J$. The other inclusion holds since $r\in C(I)$.
\end{proof}

Another consequence of \Cref{cilocal} is 

\begin{Lemma}\label{variety}
The set $\mathcal{L}(I)=\{P\in \Spec(R): I_P \textrm{ is of linear type}\}$ is an open set in $\Spec(R)$. More precisely, $\mathcal{L}(I)=\Spec(R)\setminus V(C(I))$.
\end{Lemma}

\begin{Corollary}\label{radical}
Let $I\subset R$ be an ideal in a Noetherian ring. Then $\sqrt{I}\subseteq \sqrt{C(I)}$. 
\end{Corollary}

\begin{proof}
It is enough to show that $V(C(I))\subseteq V(I)$. Suppose this is not the case. Then there exists a prime ideal $P\in V(C(I))\setminus V(I)$ and we obtain $R_P\neq C(I)_P=C(I_P)=C(R_P)=R_P[t]\cap R_P=R_P$, a contradiction. 
\end{proof}

\begin{Corollary}
 Let $R$ be a regular ring,  and  let $I\subsetneq  R$ be  a radical ideal. Then $I\neq C(I)$.
 In particular, if $I$ is a squarefree monomial ideal,  then $I\neq C(I)$.
\end{Corollary}

\begin{proof}
Let $P$ be a minimal prime ideal of $I$. Then $I_P=PR_P$,  because  $I$ is a radical ideal. By definition the ring $R_P$ is regular. Since $R_P$ is regular, $PR_P$ is generated by a regular sequence. 
Suppose that $I=C(I)$, then
 $
 I_P=C(I)_P=C(I_P)=R_P,
 $
 a contradiction. 
\end{proof}
 
\begin{Corollary}
Let $R$ be a Noetherian ring,  and let $I\subset R$ be a proper ideal. Let $P\in V(I)$ such that $\mu(I_P)>\height P$, where $\mu(I_P)$ is the number of minimal generators of $I_P$. Then $C(I)\subseteq P$.   
\end{Corollary}

\begin{proof}
Suppose that $C(I)\nsubseteq P$. Then by~\Cref{cilocal}, $C(I_P)=C(I)_P=R_P$. Hence $I_P$ is of linear type. So \cite[Proposition~2.4]{HSV} implies  that $\mu(I_P)\leq \height P$. 
\end{proof}

\begin{Proposition}\label{equalReesSym}
Let $R$ be a polynomial ring over a field $K$ with the graded maximal ideal $\mm$. Let  $I\subset R$ be a graded ideal for which $C(I)$ is $\mm$-primary. Then $\dim \Rees(I)=\dim S(I)$ if and only if $\mu(I)\leq \dim R+1$.   
\end{Proposition}

\begin{proof}
The exact sequence 
\[
0\to J/L\To S(I)\To \Rees(I)\To 0
\]
of $T$-modules implies that $\dim S(I)=\max\{\dim R+1, \dim J/L\}$. Here we used that $\dim \Rees(I)=\dim R+1$. Since the ideal $C(I)\subseteq R$ annihilates $J/L$, it follows that $\dim J/L\leq \dim T/C(I)T$. Now
$T/C(I)T\iso (R/C(I))[y_1,\dots,y_m]$,  where $m=\mu(I)$. By assumption, $\dim R/C(I)=0$. Therefore, $\dim T/C(I)T=\mu(I)$. Hence, if $\mu(I)\leq \dim R+1$,  then $\dim S(I)=\dim \Rees(I)$. 

Conversely, assume that $\dim S(I)=\dim \Rees(I)$. Then \cite[Proposition~8.1]{HSV} implies that $\mu(I_P)\leq \height P+1$ for all $P\in \Spec(R)$. For $P=\mm$, we obtain $\mu(I)\leq \dim R+1$, as desired.
\end{proof}

\begin{Remarks}
{\em 
(a) If $I$ is $\mm$-primary, then by \Cref{radical}, $C(I)$ is also $\mm$-primary. But the converse is not true in general. Indeed, for $I=(x_1^6,x_1x_2^5,x_2^2x_3^4,x_2^3x_3^3,x_4)$, one can see that $C(I)$ is $\mm$-primary, while $I$ is not.

(b) The equality $\dim \Rees(I)=\dim S(I)$ does not necessarily imply that $I$ is of linear type. Consider the ideal $I=(x_1^2,x_1x_2,x_2^2)$. Then $I$ is not of linear type. We have $C(I)=(x_1,x_2)$ and $\mu(I)=\dim R+1=3$. Therefore, by \Cref{equalReesSym}, $\dim \Rees(I)=\dim S(I)=3$.

(c) Let $(R,\mm)$ be a Noetherian local ring, or a finitely generated graded $K$-algebra which is Cohen-Macaulay and $I\subset R$ is a (graded) ideal which contains a non-zerodivisor. If we replace the polynomial ring in \Cref{equalReesSym} by any such ring, then \Cref{equalReesSym} remains valid. 
}
\end{Remarks}

\begin{Lemma}\label{mgradedIdeal}
Let $R$ be a $\ZZ^n$-graded ring, $M$ be a $\ZZ^n$-graded $R$-module and $N$ be a $\ZZ^n$-graded submodule of $M$. Then $N:_RM$ is a $\ZZ^n$-graded ideal of $R$.
\end{Lemma}
\begin{proof}
Consider an element $r\in N:_RM$. Then  $r=\sum_{a\in \ZZ^n} r_a$ such that $r_a\in R_a$ and $r_a=0$, except for finitely many $a$. We need to show that $r_a\in N:_R M$ for all $a$. Since $M$ is $\ZZ^n$-graded, it is enough to show that for any $b\in \ZZ^n$, $r_aM_b\subseteq N$. Note that $rM_b\subseteq rM\subseteq N$. Thus $\sum_{a\in \ZZ^n} r_am\in N$ for any $m\in M_b$. By the assumption on $N$ we conclude that $r_am\in N$ for all $m\in M_b$. Hence $r_aM_b\subseteq N$, as desired.   
\end{proof}

\begin{Proposition}\label{multigrading}
Let $R$ be a $\ZZ^n$-graded Noetherian ring. If $I\subset R$ is a  $\ZZ^n$-graded ideal, then $C(I)$ is a  $\ZZ^n$-graded ideal. 
\end{Proposition}

\begin{proof}
Let $\{f_1,\ldots,f_m\}$ be a generating set of $I$ consisting of $\ZZ^n$-graded elements.  
In view of \Cref{mgradedIdeal}, it is enough to show that there exists a $\ZZ^n$-grading on the polynomial ring $T = R[y_1, \ldots, y_m]$ such that $L$ and $J$ are $\ZZ^n$-graded $R$-submodules of $T$.
Consider a $\ZZ^{n}$-grading on $T$ which is induced by the $\ZZ^{n}$-grading on $R$ together with $\deg(y_i)=\deg(f_i)$ for all $i$. In other words, for $a\in\ZZ^n$,  an element of $T_a$ is of the form $ry_1^{t_1}\cdots y_m^{t_m}$ where $r\in R_b$ for some $b\in\ZZ^n$ with $b+\sum_{i=1}^m t_i\deg(f_i)=a$. Then $T$ viewed as an $R$-module is $\ZZ^n$-graded with respect to this grading, and the ideals $L,J\subset T$ are $\ZZ^n$-graded $R$-submodules of $T$, since they are generated by $\ZZ^n$-graded binomials. Therefore $C(I)=L:_R J$ is a $\ZZ^{n}$-graded ideal of $R$.  
\end{proof}

\begin{Corollary}\label{CIisMon}
If $R$ is a polynomial ring and $I$ is a monomial ideal, then $C(I)$ is a monomial ideal, and if $R$ is a graded ring and $I$ is a graded ideal in $R$, then $C(I)$ is a graded ideal.
\end{Corollary}


\section{Variations on a theorem of Valla}

It is well-known that $I$ is of linear type if $I$ is generated by a regular sequence. So it is natural to ask whether the ideal  $I'=(I,f)$ is of linear type, when  $f$ is a non-zerodivisor modulo $I$ and $I$ is of linear type. Related to this question, Valla~\cite{Valla} showed that if $I'$ is of linear type and $f$ is a non-zerodivisor modulo $I$ which lies in the
Jacobson radical of $R$, then   $I$ is of linear type. Moreover, if $I^j: f=I^j$  for all $j\geq 0$, then the converse holds as well. 

In this section we consider variations and partial generalization of  Valla's result in term  of  the conductor ideal of $I$. 

\begin{Theorem}
\label{regularmodI}
Let $I$ be a monomial ideal in the polynomial ring $K[x_1,\ldots,x_n]$, and let $u$ be a monomial which is a non-zerodivisor modulo $I$. Then 
\begin{enumerate}
\item[(a)] $C((I,u))\subseteq (C(I),u)$. 
\item[(b)] $(C(I),u^k)\subseteq C((I,u))$ for some positive integer $k$.
\end{enumerate}
  
\end{Theorem}
\begin{proof}
(a) First we show that $C((I,u))\subseteq (C(I),u)$. Without loss of generality we may assume that $I\subset S=K[x_1,\ldots,x_t]$ for some $t<n$ and $\supp(u)=\{x_{t+1},\ldots,x_n\}$. We let $T=S[y_1,\ldots,y_m]$ and $C(I) = L:_SJ $ with $L$ and $J$ as before. We set $S'=K[x_1,\ldots,x_n]$ and  $T'=T[y]$ to be polynomial rings and $I'=(I,u)$.
Define the $K$-algebra homomorphism $\varphi'\: T'\to S(I')$ with $\varphi'_{|T}=\varphi$, $\varphi'(x_i)=x_i$ for $t+1\leq i\leq n$  and $\varphi'(y)=u$. Let $J'=\Ker(\varphi')$ and $L'\subseteq J'$ the ideal generated by the linear relations in the $y_i$'s and $y$. 
By \Cref{CIisMon}, $C(I')$ is a monomial ideal. Therefore, it is enough to show that for any monomial $w\in C(I')$ which is not divisible by $u$, one has $w\in C(I)S'$.
Consider a generator $h=vy_{i_1}\cdots y_{i_s}-v'y_{j_1}\cdots y_{j_s}$ of $J$, where $v$ and $v'$ are monomials in $S$, and a monomial $w\in C(I')$ with $u\nmid w$. Since $J\subseteq J'$, we have $h\in J'$. Thus $wh\in L'$ and we may write $wh=\sum_{i=1}^kg_ih_i$ where $h_i\in L'$ are binomial generators and $g_i\in T'$. 

As was shown in the proof of \Cref{multigrading}, if we consider the natural $\ZZ^n$-grading on $S'$, then $J', L'\subset T'$ are $\ZZ^n$-graded $S'$-modules and  each binomial generator of $J'$ or $L'$ is $\ZZ^n$-graded with respect to the given grading. 
It follows that $wh$ and all $h_i$ are $\ZZ^n$-graded. So one may choose $g_i$'s to be $\ZZ^n$-graded elements, so that $\deg(wh)=\deg(g_ih_i)$ for all $i$. We claim that $h_i\in L$ for all $i$. Suppose this is not the case for some $i$. Then $h_i=u'y_j-u''y$ for some monomials $u',u''$ in $S'$ and an integer $j\in [m]$. Therefore, $\deg(h_i)=\deg(u)+\deg(u'')$, and then $\deg(wh)=\deg(g_i)+\deg(u)+\deg(u'')$. Let $\deg(w)=(a_1,\ldots,a_n)$, $\deg(h)=(a'_1,\ldots,a'_n)$, and let $u=x_{t+1}^{b_{t+1}}\cdots x_n^{b_n}$. Then $a_i+a'_i\geq b_i$ for $t+1\leq i\leq n$ and $a'_i=0$ for $t+1\leq i\leq n$. Thus $a_i\geq b_i$ for $t+1\leq i\leq n$. Hence $u$ divides $w$.  This contradicts to  our assumption on $w$. So the claim is proved and $h_i\in L$ for all $i$. By a similar argument one gets $g_i\in S'[y_1,\ldots,y_m]$ for all $i$.
Let $\deg(g_i)=(c_{1,i},\ldots,c_{n,i})$ for all $i$. Since $h,h_i\in T$ for all $i$, we conclude that $(a_{t+1},\ldots,a_n)=(c_{t+1,i},\ldots,c_{n,i})$ for all $i$. Hence $w'=x_{t+1}^{a_{t+1}}\cdots x_n^{a_n}$ divides $g_i$ for all $i$. Dividing both sides of the equality $wh=\sum_{i=1}^kg_ih_i$ by $w'$ we obtain $(w/w')h=\sum_{i=1}^kg'_ih_i$ with $g'_i\in T$. Hence $w/w'\in L:_S J=C(I)$, which implies that $w\in C(I)S'$, as desired. 

(b) We keep the notation as in (a). First we show that $C(I)\subseteq C(I')$. Let $f$ be a minimal binomial generator of $J'$. Then $f$  is $\ZZ^n$-graded. We may write 
$f=u'y^ag'-u''y^bg$ with $g$ and $g'$ monomials in $T$ and $u',u''$ monomials in $K[x_{t+1},\ldots,x_n]$. Since $f$ is a minimal generator of $J'$,  we must have that $\gcd(u'y^a, u''y^b)=1$. Hence we may assume that $f=u'g'-u''y^{b}g$. Note that $\deg(u'g')=\deg(u''y^bg)$.
This implies that $u'=u^bu''$. Since $\gcd(u',u'')=1$, we get $u''=1$ and thus $u'=u^b$. Therefore, $f=u^b g'-y^{b}g$. 
Let $g'=w'y_{i_1}\cdots y_{i_d}$, where $w'$ is monomial in $S$. Then $d\geq b$, since $y^bg$ has at least $b$ factors of the form $y$ and $y_i$. Thus
$u^b g'=w'(uy_{i_1})\cdots (uy_{i_b})y_{i_{b+1}}\cdots y_{i_d}$.
Since $uy_i=u_iy$ modulo $L'$, we have  $u^b g'=y^b g''$ modulo $L'$ with $g''$ a monomial in $T$. Hence we may assume that $f=y^b(g''-g')$ with $g''$ and $g'$ in $T$, and $g''-g'\in J'$. This is the case since $J'$ is a prime ideal and $y\not\in J'$. 
On the other hand, since $g'', g' \in T$, we have that  $g''-g'\in J$. Thus, if $v\in C(I)$ is a monomial, then $vf=y^bv(g''-g')=0$ modulo $L'$. This shows that $v\in C(I')$. 

Now, we show that $u^k\in C(I')$ for some positive integer $k$. Consider some binomial generator $f=wg-w'g'$ of $J'$, where $w,w'$ are monomials in $S'$ and $g$ and $g'$ are monomials in $K[y_1,\ldots,y_m,y]$.
The monomials $g$ and $g'$  have the same number of factors of the form $y_i$ and $y$, say $k$ factors. Then $u^kf=v_1y^k-v_2y^k=y^k(v_1-v_2)$ modulo $L'$, where $v_1, v_2$ are monomials in $S'$. Since $y^k(v_1-v_2)\in J'$ and $y^k\not\in J'$, it follows that $v_1-v_2\in J'$. This is only possible if $v_1=v_2$, and so $u^kf=0$ modulo $L'$.   Since $J'$ is finitely generated, we can find a common integer $k$ such that $u^k J'\subseteq L'$.  Therefore, $u^k\in C(I')$.
\end{proof}

For monomial ideals, the following corollary generalizes a result of Valla \cite[Theorem 2.3]{Valla}, see also  \cite[Proposition~3.11]{HSV}. 

\begin{Corollary}\label{monlinear}
With the assumptions and  notation  of \Cref{regularmodI}  it follows that $I$ is of linear type if and only if $(I,u)$  is of linear type.  
\end{Corollary}

In general, for a graded ideal $I$ of linear type in a finitely generated graded $K$-algebra, it may happen that  there is a non-zerodivisor element $f$ modulo $I$ such that $(I,f)$ is not of linear type, see~\cite[Example 2.4]{Valla}.  

\medskip
 It is possible to generalize part (b) of \Cref{regularmodI} by adding some extra assumptions. This will be done in the next result. Before doing that, we will need the following concept: let $R$ be Noetherian and $I\subset R$ be an ideal. 
There is an $R$-linear map $\partial: S(I)\to S(I)(-1)$ which  assigns to $c_1c_2\cdots c_k \in S_k(I)$ with $c_i\in S_1(I)=I$ the element $c_1\cdot(c_2c_3\cdots c_k)\in S_{k-1}(I)$. Here $\cdot: (R, S(I))\to S(I)$  is the multiplication given by the $R$-module structure of $S(I)$. The definition of the map does not depend on which factor is pulled out and can be linearly extended to $S(I)$, see \cite{HSV}. Note that $\partial^k$ applied to $S_k(I)$ coincides with the natural $R$-module homomorphism $S_k(I)\to I^k$. Hence,  if   $\alpha\: S(I)\to\Rees(I)$ denotes the natural $R$-algebra homomorphism, then $\alpha(c)=\partial^k(c)t^k$ for all  $k$ and all $c\in S_k(I)$. 

\begin{Theorem}
\label{generalizationvalla}
Let $R$ be a Noetherian ring, $I\subset R$ be an ideal,  and let $f\in R$ be an element such that $I^j:f=I^j$ for all $j\geq 0$. Then there exists an integer $k>0$ such that  $(C(I), f^k)\subseteq C((I,f))$.
\end{Theorem}

\begin{proof}
Since $I:f=I$, we obtain  an exact sequence
\[
0\To M\To I\dirsum R\stackrel{\varphi}\To I'\to 0,
\]
where $\varphi(a,b)=a-bf$, $M=\{(bf,b)\: b\in I\}$ and $I'=(I,f)$.

Then this exact sequence gives rise to the following commutative diagram with exact rows
 \begin{center}
\begin{tikzcd}
0\arrow{r}&G\arrow{r} \arrow[d,"\gamma"] &S(I)[y]\arrow[r,"\tau"] \arrow[ d, "\beta"]& S(I')\arrow{r}\arrow[ d, "\alpha' "] & 0\\
0\arrow{r}&H\arrow[r]&\Rees(I)[y]\arrow[r,"\delta"]& \Rees(I')\arrow{r}&0
\end{tikzcd}
\end{center}
Here,  the top exact row is induced by the above short exact sequence, $\beta=\alpha[y]$,    $\alpha\: S(I)\to \Rees(I)$
and $\alpha'\: S(I')\to \Rees(I')$ are the natural 
epimorphisms. Furthermore, $\delta|_{\Rees(I)}$  is the natural inclusion map $\Rees(I)\subseteq  \Rees(I')$, and $\delta(y) =ft$. Finally,  $H=\Ker(\delta)$, and  $\gamma$ is induced by $\beta$. 

Note that for all integers $k>0$, we have $fs-\partial(s)y\in G$ for all $s\in S_k(I)$. Moreover, $G$ is generated by the elements $fs-\partial(s)y$ with $s\in S_1(I)$. This implies that $\Im(\gamma)$ is generated by the elements $f(bt)-by$ with $b\in I$. 

We first  prove that $C(I)\subseteq C((I,f))$. The commutative diagram induces the short exact sequence 
\[
\Ker(\beta)\to \Ker(\alpha')\to \Coker(\gamma)\to 0, 
\]
Obviously, $\Ker(\beta)=\Ker(\alpha)[y]$. Therefore, $\Ann_R(\Ker(\beta))=C(I)$. 
Next we will show that $\Coker(\gamma)=0$. Then this will imply that $C(I)\subseteq C(I')$.  

Since  $\Coker(\gamma)$ is a graded $R$-module,  it suffice to show that  $\Coker(\gamma)_k=0$ for all $k\geq 0$. 
The assertion is trivial for $k=0$, since $H_0=0$. Therefore, we may assume that $k>0$, and we  let $h\in H_k$. Then, there exist $a_j\in I^{k-j}$ such that 
\[
h=\sum_{j=0}^k(a_jt^{k-j})y^j.
\]
By induction on $\deg_yh$, we show that $h\in \Im(\gamma)$. 
Since $h\in \Ker(\delta)$, it follows that 
\begin{eqnarray}\label{ytof}
    \sum_{j=0}^ka_jf^j=0.
\end{eqnarray}
If $\deg_y h=0$,  then $a_0=0$, and hence $h=0$. Suppose now that $\deg_y h=s$ with $s>0$. Then \eqref{ytof} implies that $a_sf^s=-\sum_{j=0}^{s-1}a_jf^j$. The element on the  right hand side of this equation belongs to $I^{k-s+1}$. Since we assume that $I^{k-s+1}:f=I^{k-s+1}$ we see that $a_s\in I^{k-s+1}:f^s=I^{k-s+1}:f=I^{k-s+1}$. Thus we can write $a_s=\sum_lb_lc_l$ with $b_l\in I^{k-s}$ and $c_l\in I$.  Therefore, modulo $\Im(\gamma)$ we have
\[
a_st^{k-s}y^s =\sum_lb_lt^{k-s}(c_ly)y^{s-1}
=\sum_lb_lt^{k-s}(fc_lt)y^{s-1}=fa_st^{k-s+1}y^{s-1}.
\]
Hence,  modulo $\Im(\gamma)$,  $h$ is congruent  to an element $h'\in \Ker(\delta)$ with $\deg_y h'\leq s-1$. Our induction hypothesis implies that $h'\in \Im(\gamma)$, and therefore  $h\in \Im(\gamma)$, as well. This shows that $\gamma$ is surjective, as desired. 

Next we show that if $h\in \Ker(\alpha')_k$ for some $k>0$,  then $f^kh=0$. Hence, since $\Ker(\alpha')$ is finitely generated, there exists an integer $k$ such that $f^k\Ker(\alpha')=0$. In other words,  $f^k\in C((I,f))$ for this integer.

Let $h\in \Ker(\alpha')_k$.  We have seen that $\Ker(\beta)\to \Ker(\alpha')$ is surjective. Therefore, there exists $g\in S(I)[y]_k$ such that $h=\tau(g)$. Let $g=\sum_{j=0}^ka_jy^{k-j}$ with $a_j\in S(I)_j$ for all $j$. Since $h\in \Ker(\alpha')$, it follows that 
\begin{eqnarray*}
0=\alpha'(h)=\delta(\beta(g))=\sum_{j=0}^k\alpha(a_j)(ft)^{k-j}=(\sum_{j=0}^k\partial^j(a_j)f^{k-j})t^k.
\end{eqnarray*}
Therefore, 
\begin{eqnarray}
\label{inker}
\sum_{j=0}^k\partial^j(a_j)f^{k-j}=0.
\end{eqnarray}
Now $f^kh=\sum_{j=0}^kf^{k-j}\tau(f^ja_j)\tau(y)^{k-j}$. Since  $f^ja_j-\partial^j(a_j)y^j\in \Ker(\tau)$, we see that $\tau(f^ja_j)=\partial^j(a_j)\tau(y)^j$. Thus, together \eqref{inker} it follows that   
\[f^kh=\sum_{j=0}^kf^{k-j}\partial^j(a_j)\tau(y)^j\tau(y)^{k-j}=(\sum_{j=0}^kf^{k-j}\partial^j(a_j))\tau(y)^k =0,
\]
as desired. 
\end{proof}

Let $(R,\mm)$ be a local ring. We say that the ideal $I$ is of {\em linear type on the punctured spectrum}, if $I_P$ is of linear type for any prime ideal $P\neq \mm$.

\begin{Corollary}
\label{flocal}  In addition to the assumptions of \Cref{generalizationvalla},  assume that $R$ is local with maximal ideal $\mm$ and that $f\in\mm$. If  $I$ is of linear type on the punctured spectrum, then $(I,f)$ is of linear type on the punctured spectrum. 
\end{Corollary}

\medskip





\section{The conductor of special classes of monomial ideals}

We will now determine the conductor ideal for two families of monomial ideals. From \Cref{CIisMon} we know that $C(I)$ is a monomial ideal in this case, and this fact will be used throughout this section. Let $R=K[x_1,\dots, x_n]$. The monomial ideals considered here will be of \emph{fiber type}. That is, let $\varphi':K[y_1,\dots, y_m]\to K[f_1t,\dots, f_mt]$ be defined by $\varphi'(y_i)=f_it$ and set $H=\Ker{\varphi'}$. Then $I$ is of fiber type if $J=L+HT$. Here, as before, $J$ is the defining ideal of the Rees algebra and $L$ is the ideal generated by linear relations. Then $C(I)=(L:_T J)\cap R=(L:_T HT)\cap R$. Finally, when $I$ is a monomial ideal, then $L$, $H$ and $J$ are all binomial ideals. When working with these binomial ideals, we will often use the following. 
\begin{Lemma} [{\cite[Lemma 3.8]{Binomial_ideals}}] \label{Substitutions}
Let $I\subset K[x_1,\ldots, x_n]$ be an ideal generated by binomials $f_1,\ldots,f_r$. Let $\bm{x^u}-\bm{x^v}$ be a binomial belonging to $I$. Then, there exists an expression
$$
\bm{x^u}-\bm{x^v}=\sum_{k=1}^{s}\epsilon_k \bm{x^{w_k}}f_{i_k},
$$
where $\epsilon_k\in\{\pm 1\}$, $\bm{w_k}\in \mathbb{Z}_{\ge 0}^n$, and $1\le i_k\le r$ for $k=1,2,\ldots,s$, and where $\bm{x^{w_p}}f_{i_p}\ne \bm{x^{w_q}}f_{i_q}$ for all $1\le p<q\le s$.
\end{Lemma}
Another way to phrase the above lemma is to say that there is a chain of monomials  $\bm{x^u}=\bm{x^{a_1}}\to \bm{x^{a_2}}\to\ldots\to \bm{x^{a_s}}\to\bm{x^{a_{s+1}}} =\bm{x^v}$ (of length $0$ if $\bm{x^u}=\bm{x^v}$) such that $\bm{x^{a_k}}-\bm{x^{a_{k+1}}}=\epsilon_k\bm{x^{w_k}}f_{i_k}$. That is, each arrow corresponds to a substitution from one monomial to the next one using a generator of the binomial ideal. The condition $\bm{x^{w_p}}f_{i_p}\ne \bm{x^{w_q}}f_{i_q}$ for all $1\le p<q\le s$ is equivalent to saying that this chain can be chosen to be minimal in a sense that all monomials in the chain are distinct unless $\bm{x^u}=\bm{x^v}$. We will later use this lemma to prove that a given binomial does not belong to a specified binomial ideal.


\subsection{Bounded Veronese ideals}
In this subsection we will compute the conductor ideal of the bounded Veronese ideal $I_{n,d,c}$ for all values of $n$, $d$ and $c$. Recall that $I_{n,d,c}$ is generated by all $x_1^{\alpha_1}\cdots x_n^{\alpha_n}$ such that $\alpha_1+\cdots+\alpha_n=d$ and $\alpha_i\le c$ for all $i$. 

First of all, we will assume $n\ge 2$, $d\ge 0$, $c\ge 0$ since the case $n=1$ is trivial. Secondly, we may assume that $c\le d \le nc$. Indeed, if $d<c$, then $I_{n,d,c}=I_{n,d,d}$ (which falls into the case $c\le d \le nc$), whereas if $d>nc$, then $I_{n,d,c}=0$ and this case is trivial. 
Before doing further case reductions, we will need the following lemma.

\begin{Lemma}
\label{Lem:Rees algebra of a scaled ideal}    
Let $R$ be a Noetherian ring, $I$ and ideal of $R$ and $f\in R$ a non-zerodivisor with $I'=fI$. Moreover, let $J$ and $J'$ be the defining ideals of $\Rees(I)$ and $\Rees(I')$ respectively. Then $J=J'$.
\end{Lemma}

\begin{proof}
Say $I=(f_1,\dots, f_m)$ and let $T=R[y_1,\dots, y_m]$ be equipped with the map $\varphi:T\to \Rees(I)$ as before by setting $\varphi(y_i)=f_it$ so that $J=\Ker \varphi$ and similarly for $I'$. The key thing to note here is that $J\subset T$ is a homogeneous ideal. So pick an $h\in J$, homogeneous of degree $d$. Then
\[
0=f^d\varphi(h) = f^dh(y_1t,\dots, y_mt) = h(fy_1t,\dots, fy_mt) = \varphi'(h),
\]
showing that $J\subseteq J'$. But reading the equalities from right to left, we get that if $\varphi'(h)=0$, then $f$ being a non-zerodivisor forces $\varphi(h)=0$ and we get the reverse inclusion $J'\subseteq J$ as desired.
\end{proof}

Finally, given $c\le d\le nc$ and $n\ge 2$, we have either $c\le d\le (n-1)c$ or $(n-1)c<d\le nc$.  In the latter case, we can do a reduction to the former one using the following lemma. 
\begin{Lemma}
\label{Lem:reducing Veronese}
Let $(n-1)c<d\le nc$. Then $I_{n,d,c}=(x_1\cdots x_n)^{d-(n-1)c}I_{n,(nc-d)(n-1), nc-d}$.   
\end{Lemma}
\begin{proof}
 In order to prove the claim, we first let $k=d-(n-1)c\ge 1$. Then $c-k=c-d+(n-1)c=nc-d$ and $d-nk=d-n(d-(n-1)c)=d-nd+n^2c-nc=(nc-d)(n-1)$. Therefore the statement can be rewritten as $I_{n,d,c}=(x_1\cdots x_n)^kI_{n,d-nk,c-k}$. First let $m=x_1^{\alpha_1}\cdots x_n^{\alpha_n}$ be a minimal generator of $I_{n,d-nk,c-k}$ so that $\alpha_1+\ldots+ \alpha_n=d-nk$ and $\alpha_i\le c-k$ for all $i$. Then let $m'=(x_1\cdots x_n)^{k}m=x_1^{\beta_1}\cdots x_n^{\beta_n}$, where $\beta_i=\alpha_i+k\le c-k+k=c$ for all $i$, and $\beta_1+\ldots+\beta_n=\alpha_1+\ldots+\alpha_n+nk=d-nk+nk=d$. In other words, $m'=(x_1\cdots x_n)^{k}m$ is a minimal generator of $I_{n,d,c}$. Conversely, let $m'=x_1^{\beta_1}\cdots x_n^{\beta_n}$ be a minimal generator of $I_{n,d,c}$. We will first show that $\beta_i\ge k$ for all $i$. Indeed, assume without loss of generality that $\beta_1\le k-1$. We also have $\beta_i\le c$ for all $2\le i\le n$. Then $\beta_1+\ldots+\beta_n\le k-1+(n-1)c=d-(n-1)c-1+(n-1)c=d-1$, which is a contradiction. Therefore, $\beta_i\ge k$ for all $i$ and we can write $m'=(x_1\cdots x_n)^km$ where $m=x_1^{\alpha_1}\cdots x_n^{\alpha_n}$ with $\alpha_i=\beta_i-k\le c-k$ for all $i$, and $\alpha_1+\ldots+\alpha_n=\beta_1+\ldots+\beta_n-nk=d-nk$. Therefore, $m$ is a minimal generator of $I_{n,d-nk,c-k}$. This finishes the proof.   
\end{proof}

Combining \Cref{Lem:Rees algebra of a scaled ideal} and \Cref{Lem:reducing Veronese} we conclude that it suffices to consider the case $c\le d\le (n-1)c$.
Note that it is possible that $nc-d=0$, in which case $I_{n,d,c}$ reduces to $I_{n,0,0}$. This is the reason we could not discard the case of $c=0$ until this point. But now we can finally do so: since we are in the case $c\le d\le (n-1)c$, the case $c=0$ implies $d=0$, and therefore $I_{n,d,c}=I_{n,0,0}=R$. In conclusion, from now on we can assume $1\le c\le d\le (n-1)c$.

\begin{Theorem}
\label{veroneseConductor}
Let $I_{n,d,c}$ be an ideal of Veronese type where $1\le c\leq d \leq (n-1)c$. Then $I_{n,d,c}$ is of linear type if $(d,c)=(n-1,1)$ or $(d,c)=(1,1)$ and otherwise
\[
C(I_{n,d,c}) = I_{n,p+1,1} \text{ where } p=\left \lfloor \frac{d-2}{c} \right \rfloor.
\]
\end{Theorem}

\begin{proof}
We begin by noting that $I_{n,d,c}$ is a polymatroidal ideal with the strong exchange property. Hence the ideal is of fiber type (or possibly even of linear type) and its ideal of fiber relations $H$ is generated by symmetric exchange relations \cite[Theorem 5.3]{Discrete_Poly}. Consider such a non-trivial symmetric exchange relation $y_iy_j-y_ky_l\in H$ corresponding to monomials $m_i$, $m_j$, $m_k$, $m_l$, where $\deg_s(m_i)>\deg_s(m_j)$, $\deg_t(m_j)>\deg_t(m_i)$, $m_k=\frac{x_t}{x_s}m_i$ and $m_l=\frac{x_s}{x_t}m_j$. Note that if $(d,c)=(n-1,1)$, then the exchange relation are forcing that $m_i=m_l$ and $m_j=m_k$. Hence there are no non-trivial fiber relations in this case and so $I_{n,n-1,1}$ is of linear type as stated. Similarly, if $(d,c)=(1,1)$, then there are no non-trivial exchange relations and the ideal is of linear type. This also follows since $I_{n,1,1}$ is generated by a regular sequence. 

Assume from now on that we are in the case $(d,c)\neq (n-1,1)$ and $(d,c)\neq (1,1)$. Let us prove that $I_{n,p+1,1}\subseteq C(I_{n,d,c})$. By definition of $p$, we can write $d-2=cp+r$ with $0\le r\le c-1$. Let $u\in I_{n,p+1,1}$ and $0\neq y_iy_j-y_ky_l\in H$ be a fiber relation. Note that $u^c$ cannot divide both $m_i$ and $m_j$. Indeed, this would force that $\deg\gcd(m_i,m_j)\ge c(p+1)=cp+c\ge cp+r+1=d-1$. In other words, $\deg\gcd(m_i,m_j)\in\{d-1,d\}$, which is a contradiction since $m_i$ and $m_j$ cannot give rise to a non-trivial fiber relation in that case. Therefore there is some variable $x_v\in \supp u$ such that $x_v^c$ does not divide both $m_i$ and $m_j$, say it does not divide $m_i$. Writing $y_{i,s/t}$ for the variable corresponding to $\frac{x_s}{x_t}m_i$, we have 
\[
uy_iy_j = \frac{ux_s}{x_v}y_{i,v/s}y_j = \frac{ux_t}{x_v}y_{i,v/s}y_{j,s/t} =uy_{i,t/s}y_{j,t/s} = uy_ky_l \pmod L.
\]
Hence $u(y_iy_j-y_ky_l) \in L$ and $I_{n,p+1,1} \subseteq C(I_{n,d,c})$. 

Recall that as $I_{n,d,c}$ is a monomial ideal, we know from \Cref{CIisMon} that $C(I_{n,d,c})$ is also a monomial ideal. Thus for the other inclusion it is enough to show that any monomial $u$ with $|\supp u|\leq p$ is not in $C(I_{n,d,c})$.  By symmetry, it is enough to show that $u=x_1^{\gamma_1}\cdots x_p^{\gamma_p}\not\in C(I_{n,d,c})$ for any choice of $(\gamma_1,\ldots, \gamma_p)$. 

We now claim that if there is some non-trivial fiber relation $y_iy_j-y_ky_l$ where all $y$ correspond to monomials divisible by $(x_1\cdots x_p)^c$, then $u(y_iy_j-y_ky_l)\notin L$. Assume for a contradiction that $u(y_iy_j-y_ky_l)\in L$. Then by \Cref{Substitutions} and given that $L$ is bigraded, there should exist a non-empty chain of monomials $uy_iy_j\to u_1y_{i_1}y_{j_1}\to\ldots\to uy_ky_l$. As before, let $m_{i_1}$ and $m_{j_1}$ we the monomials corresponding to $y_{i_1}$ and $y_{j_1}$, respectively. 
We have $um_im_j=u_1m_{i_1}m_{j_1}$, and comparing $x_s$-degrees on both sides for any $1\le s \le p$ we obtain
$$\deg_{x_s}u+2c=\deg_{x_s}u_1+\deg_{x_s}m_{i_1}+\deg_{x_s}m_{i_2}\le \deg_{x_s}u_1+2c
$$
and thus $\deg_{x_s}u\le \deg_{x_s}u_1$ for all $1\le s \le p$. Since $\deg u=\deg u_1$ and $\supp u \subseteq \{1,2,\dots, p\}$, we conclude that $u=u_1$. Now let $\ell=m_1y_{i'}-m_2y_{j'}$ be a non-zero linear relation satisfying $uy_iy_j- u_1y_{i_1}y_{j_1}=\pm m \ell$ for some monomial $m$ in $T$ coming from $uy_iy_j\to u_1y_{i_1}y_{j_1}$. Clearly we have $m_1\ne m_2$. Therefore $u(y_iy_j-y_{i_1}y_{j_1})=\pm m(m_1y_{i'}-m_2y_{j'})$ and setting all $y$ to $1$ on both sides, we obtain a contradiction. In other words, no substitutions can be made starting at a monomial where all $y$ correspond to monomials divisible by $(x_1\cdots x_p)^c$.

We will now give examples of such fiber relations $y_iy_j-y_ky_l$ where all of the variables does correspond to monomials divisible by  $(x_1\cdots x_p)^c$. Then \Cref{Substitutions} gives us that $u(y_iy_j-y_ky_l)\notin L$ and hence $u\notin C(I_{n,d,c})$ as desired. For the following cases, recall here that $d-2=cp+r$ with $0\le r\le c-1$.\\ 
Case 1: $c=1$. In this squarefree case we have that $p=d-2$, and as $d\leq n-2$, we get $p\leq n-4$. This allows us to take $m_i=x_1\cdots x_d$, $m_j=x_1\cdots x_{d-2}x_{d+1}x_{d+2}$, $m_k=x_1\cdots x_{d-2}x_{d-1}x_{d+1}$ and $m_l=x_1\cdots x_{d-2}x_{d}x_{d+2}$.\\
Case 2: $c\geq 2$ and $r\leq c-2$. Note that $p\leq n-2$ as $d\leq (n-1)c$. Here we can take $m_i=x_1^c\cdots x_p^cx_{p+1}^{r+2}$, $m_j=x_1^c\cdots x_p^cx_{p+1}^{r}x_{p+2}^2$ and $m_k=m_l=x_1^c\cdots x_p^cx_{p+1}^{r+1}x_{p+1}$ to get the sought relation.\\
Case 3: $c\geq 2$ and $r=c-1$. In this case we note that $cp=d-2-(c+1)\leq (n-1)c - (c+1)=(n-2)c-1$, so $p\leq n-3$. Hence we can use $m_i=x_1^c\cdots x_p^cx_{p+1}^{c-1}x_{p+2}^2$, $m_j=x_1^c\cdots x_p^cx_{p+1}^{c-1}x_{p+3}^2$ and $m_k=m_l=x_1^c\cdots x_p^cx_{p+1}^{c-1}x_{p+2}x_{p+3}$. Thus we have shown that any $u$ with $|\supp u|\leq p$ is not in $C(I_{n,d,c})$, so $C(I_{n,d,c})\subseteq I_{n,p+1,1}$ and we are done.
\end{proof}

By noting that $p+1=d-1$ in the squarefree case, we get the following characterization for squarefree Veronese ideals.

\begin{Corollary}
Let $I_{n,d,1}$ be the squarefree Veronese ideal of degree $d$. Then $I_{n,d,1}$ is of linear type if $d\leq 1$ or $d\geq n-1$, while for $2\leq d \leq n-2$ we have that
\[
C(I_{n,d,1})=I_{n,d-1,1}.
\]
\end{Corollary}

Following the reductions made, we can also see exactly in what cases these bounded Veronese ideals are of linear type.

\begin{Corollary}\label{BVlinear}
The bounded Veronese ideal $I_{n,d,c}$ is of linear type if and only if $n=1$ or $d\leq 1$ or $d\geq nc-1$.
\end{Corollary}

\subsection{Edge ideals of graphs with loops}

Let $G$ be a graph on the vertex set $[n]=\{1,\ldots,n\}$ and edge set $E(G)$. We allow loops on the vertices, but not parallel edges. Let $I(G)$ denote the edge ideal of such a graph, namely, $I(G)=(x_{j_k}x_{j_{k+1}}\mid \{j_k,j_{k+1}\}\in E(G))$. This need not be a squarefree ideal since $x_i^2\in I(G)$ if there is a loop at vertex $i$. Note that this edge ideal construction gives a bijection between quadratic monomial ideals and graphs which may have loops, but not parallel edges. We wish to describe $C(I(G))$ for an arbitrary such graph $G$. Before this, we need some preliminary notions.

An \emph{even closed walk of length $2q>0$} of $G$ is a non-empty sequence of edges of $G$ of the form
$$\Gamma=\{\{j_1,j_2\}, \{j_2,j_3\}, \ldots, \{j_{2q-1},j_{2q}\}, \{j_{2q},j_1\}\}.$$

Let $\Gamma$ be an even closed walk. We may write it as 
$$\Gamma=\{e_{i_1}, e_{i_2}, \ldots, e_{i_{2q}}\},$$
where each $e_{i_k}\in E(G)$. By $h_\Gamma$ we denote the binomial 
$$h_\Gamma=\prod_{k=1}^{k=q}y_{i_{2k-1}}-\prod_{k=1}^{k=q}y_{i_{2k}}.$$
Sometimes we will write $h_\Gamma=h_\Gamma^+-h_\Gamma^-$, where $h_\Gamma^+=\prod_{k=1}^{k=q}y_{i_{2k-1}}$ and $h_\Gamma^-=\prod_{k=1}^{k=q}y_{i_{2k}}$.
An even closed walk $\Gamma$ of $G$ is called \emph{primitive} if there exists no even closed walk $\Gamma'$ with $h_\Gamma\not=h_{\Gamma'}$ such that $h_{\Gamma'}^+|h_\Gamma^+$ and $h_{\Gamma'}^-|h_\Gamma^-$. An even closed walk $\Gamma$ is called \emph{trivial} if $h_\Gamma=0$.

\begin{Proposition}
\label{prop: Fiber for graph}
Let $G$ be a graph as above. Then the fiber relations of $\Rees(I(G))$ are generated by the non-trivial primitive even closed walks of $G$.
\end{Proposition}

This proposition is a classical result if $G$ is a simple graph (see \cite{Even_walks} or \cite[Lemma~5.10]{Binomial_ideals}). In fact, the proof in the latter reference works in the more general case when $G$ is allowed to have loops.

Let $\bf{\Gamma}$ denote the set of non-trivial  primitive even closed walks of $G$. For any $\Gamma \in \bf{\Gamma}$, let 
$$N_\Gamma=(x_i\mid i \text{ is adjacent to some vertex in } \Gamma)\subset R.$$
Note that in particular variables corresponding to vertices in $\Gamma$ belong to $N_\Gamma$.
Before we can proceed with the conductor of $I(G)$, we need a small lemma.
\begin{Lemma}
\label{Lem: lindegs}
Let $I=(u_1,\ldots, u_m)$ such that $\deg(u_j)=d$ for all $j=1,\ldots, m$. Then for any minimal generator $l$ of $L$ we have $\deg_x(l)\le d$. 
\end{Lemma}
\begin{proof}
Let $l=m_py_q-m_ry_s$ be a minimal generator of $L$ with $\deg m_p=\deg m_r\ge d+1$. Here, as before, $\varphi(y_j)=u_jt$ for all $j=1,\ldots, m$. Since $L\subseteq J$, we in particular have $\varphi(m_py_q)=\varphi(m_ry_s)$ and so $m_pu_q=m_ru_s$. If we assume $\gcd(m_p,m_r)=1$, then we get $m_p|u_s$, which is impossible by degree reasons. Therefore, we conclude $\gcd(m_p,m_r)\neq 1$. Hence $l'=\frac{l}{\gcd(m_p,m_q)}=\frac{m_p}{\gcd(m_p,m_q)}y_q-\frac{m_q}{\gcd(m_p,m_q)}y_s\in L$ and $l$ is not needed as a minimal generator of $L$.
\end{proof}

With this preparation we can now determine the conductor $C(I(G))$. 
\begin{Theorem}
\label{Graph_conductor}
Let $G$ be a graph which may have loops, but not parallel edges. Let $I(G)$ be its edge ideal. Then the conductor $C(I(G))$ of $I(G)$ is given by
\[
C(I(G))=I(G) + \bigcap_{\Gamma \in \bf{\Gamma}}N_{\Gamma}.
\]
\end{Theorem}

\begin{proof}
We know from \cite[Theorem~8.2.1]{Villarreal} that $I(G)$ is of fiber type when $G$ is a simple graph. This result was later noted in \cite[Theorem~3.2]{Fiber_type} to extend  to the case when $G$ is allowed to have loops. Moreover, from \Cref{prop: Fiber for graph} we know that $H$ is generated by binomials which correspond to non-trivial primitive even closed walks of $G$. Let $\Gamma \in \bf{\Gamma}$ be such a walk and $h_{\Gamma}$ the corresponding fiber relation. We want to determine $(L:_Rh_{\Gamma})$. Write 
\[
h_{\Gamma} = y_{i_1}y_{i_3}\cdots y_{i_{2q-1}} - y_{i_2}y_{i_4}\cdots y_{i_{2q}}
\]
where $y_{i_k}$ corresponds to the edge $x_{j_k}x_{j_{k+1}}$ for all $k=1,\ldots, 2q$ with the convention $x_{j_{2q+1}}=x_{j_1}$. 


Let $x_s\in N_{\Gamma}$. By cyclically permuting our labels, we may assume that it is adjacent to $x_{j_1}$. Let $y_s$ correspond to this edge. Using the linear relations 
\[
x_{j_{k-1}}y_{i_k}-x_{j_{k+1}}y_{j_{k-1}}, \; x_sy_{i_1}-x_{j_2}y_s, \text{ and } x_{j_{2q}}y_s=x_sy_{i_{2q}}
\]
we get that
\begin{align*}
x_sy_{i_1}y_{i_3}\cdots y_{i_{2q-1}} &= y_sx_{j_2}y_{i_3}\cdots y_{i_{2q-1}} = y_sy_{i_2}x_{j_4}y_{i_5}\cdots y_{i_{2q-1}} 
= \dots \\ &= y_sx_{j_{2q}}y_{i_2}y_{i_4}\cdots y_{i_{2q-2}} = x_sy_{i_{2q}}y_{i_2}y_{i_4}\cdots y_{i_{2q-2}} \pmod L.
\end{align*}
Hence $x_sh_{\Gamma}\in L$ and $N_{\Gamma} \subseteq (L:_R h_{\Gamma})$. 

Finally, if $x_sx_{t}\in I(G)$ with $y_s$ corresponding to that edge, then linear relations of the form
\[
x_sx_ty_{i_k}-x_{j_k}x_{j_{k+1}}y_s \text{ and }x_{j_{k-1}}y_{i_k}-x_{j_{k+1}}y_{j_{k-1}}
\]
give
\begin{align*}
x_{s}x_ty_{i_1}y_{i_3}\cdots y_{i_{2q-1}} &= y_{s}x_{j_1}x_{j_2}y_{i_3}\cdots y_{i_{2q-1}} = y_{i_2}y_sx_{j_1}x_{j_4}y_{i_5}\cdots y_{i_{2q-1}} = \cdots \\
&= y_{i_2}y_{i_4}\cdots y_{i_{2q-2}}x_{j_1}x_{j_{2q}}y_s = y_{i_2}y_{i_4}\cdots y_{i_{2q-2}}x_{s}x_ty_{i_{2q}} \pmod{L}.
\end{align*}
Hence $x_sx_th_{\Gamma}\in L$ and we conclude that $N_{\Gamma}+I(G) \subseteq (L:_Rh_{\Gamma})$.


Next we want to show that the above containment is an equality. Note that since $I(G)$ is a monomial ideal, we know from \Cref{mgradedIdeal} that $(L:_R h_{\Gamma})$ is also a monomial ideal. Therefore, it is enough to show that no monomial $u$ outside of $N_{\Gamma}+I(G)$ belongs to $(L:_R h_{\Gamma})$. To this end, note that $L$ is generated by two types of relations. The first type is relations of the form $x_iy_j-x_ky_i$, where $y_i$ corresponds to the edge $x_ix_j$ and $y_j$ to the edge $x_jx_k$. The second type is relations of the form $x_ix_jy_k-x_kx_ly_i$, where $y_i$ corresponds to $x_ix_j$ and $y_k$ to $x_kx_l$. It is clear from \Cref{Lem: lindegs} that there are no other minimal generators in $L$. 
Now, if we have a monomial $u$ not divisible by any monomial in $N_{\Gamma}+I(G)$, then there is no chain of substitutions that can start at $uy_{i_1}y_{i_3}\cdots y_{i_{2q-1}}$ and end at $uy_{i_2}y_{i_4}\cdots y_{i_{2q}}$. Thus $uh_{\Gamma}$ is not in $L$ by \Cref{Substitutions} and we get the reverse inclusion $(L:_Rh_{\Gamma}) \subseteq N_{\Gamma}+ I(G)$ needed for equality.

With this, we can now finish up the proof as follows. 
Write $J=L+H$ where we have $H=\sum_{\Gamma\in \bf{\Gamma}}{(h_\Gamma)}$. Then,
\[
C(I(G))=L:_R J=L:_R H=\bigcap_{\Gamma\in \bf{\Gamma}}{(L:_R h_\Gamma)}.
\]
Now using $(L:_Rh_{\Gamma}) = N_{\Gamma} + I(G)$, we get
\[
C(I(G))=\bigcap_{\Gamma\in \bf{\Gamma}}{(I(G)+N_\Gamma)}=I(G)+\bigcap_{\Gamma\in \bf{\Gamma}}{N_\Gamma},
\]
where the last equality holds because all ideals in question are monomial. Hence the proof is complete.
\end{proof}

\begin{Example}
{ \em
Let $G$ be the graph illustrated in \Cref{fig:graph}. Then $N_{\Gamma_1}=(x_1,x_2,x_3,x_4,x_5)$ and $N_{\Gamma_2}=(x_5,x_6,x_7,x_8,x_9,x_{10})$, where $\Gamma_1$ is any non-trivial primitive even closed walk on the vertices $\{1,2,3,4\}$, and $\Gamma_2$ is any such walk on the vertices $\{6,7,8,9,10\}$. As there are no other non-trivial primitive even closed walks in $G$, by \Cref{Graph_conductor} we get that
\[
C(I(G))=I(G) + N_{\Gamma_1}\cap N_{\Gamma_2} = I(G) + (x_1,x_2,x_3,x_4,x_5)\cap (x_5, x_6, x_7, x_8, x_9, x_{10}).
\]
Hence $x_5\in C(I(G))$ and all other generators are quadratic.

\begin{figure}[h]
\centering
\begin{tikzpicture}[scale=2]

\node [{draw, very thick, circle, minimum size=1cm}](n1) at (0,0) {$1$};
\node [{draw, very thick, circle, minimum size=1cm}](n4) at (1,0) {$4$};
\node [{draw, very thick, circle, minimum size=1cm}](n2) at (0,-1) {$2$};
\node [{draw, very thick, circle, minimum size=1cm}](n3) at (1,-1) {$3$};
\node [{draw, very thick, circle, minimum size=1cm}](n5) at (2,0) {$5$};
\node [{draw, very thick, circle, minimum size=1cm}](n6) at (3,0) {$6$};
\node [{draw, very thick, circle, minimum size=1cm}](n7) at (2.5,-1) {$7$};
\node [{draw, very thick, circle, minimum size=1cm}](n8) at (3.5,-1) {$8$};
\node [{draw, very thick, circle, minimum size=1cm}](n9) at (2.5,1) {$9$};
\node [{draw, very thick, circle, minimum size=1cm}](n10) at (3.5,1) {$10$};

\draw[very thick] (n4)--(n1)--(n2)--(n3)--(n4)--(n5)--(n6)--(n7)--(n8)--(n6)--(n9)--(n10)--(n6);

\end{tikzpicture}
\caption{}
\label{fig:graph}
\end{figure}
}

\end{Example}

\begin{Corollary}
\label{primary_graph}
For a simple graph $G$, the following are equivalent.
\begin{enumerate}
\item $C(I(G))$ is $\mathfrak{m}$-primary.
\item $C(I(G))=\mathfrak{m}$.
\item $G$ has a non-trivial even closed walk and any vertex in $G$ is adjacent to every non-trivial even closed walk of $G$.
\end{enumerate}
\end{Corollary}

Note that \Cref{primary_graph} is no longer true for non-simple graphs. For example, the graph with edge ideal $I=(x_1^2, x_2^2, x_3^2, x_4^2, x_1x_2, x_3x_4)$ has $C(I)=\mathfrak{m}^2$. 

\section{Questions and open problems}

There are still a lot of questions left that one can ask about the conductors of ideals. We will end by asking some of them.

\begin{Question}
{\em 
Let $G_{n,d,r}$ be the ideal in $k[x_1,\dots, x_n]$ generated by $r$ generic forms of degree $d$. Then computer experiments in \texttt{Macaulay2} \cite{M2} suggest that $C(G_{n,d,r})$ is of linear type for $r\leq n$, and for $r>n$ we have
\[
C(G_{n,d,r}) = \mathfrak{m}^{f(n,d,r)} 
\]
for some function $f(n,d,r)$ taking positive integer values. Is this the case, and can a formula for $f(n,d,r)$ moreover be given?
}
\end{Question}

\begin{Question}
{\em
In both the cases of edge ideals of simple graphs and squarefree Veronese ideals, we have seen that conductors of such squarefree ideals are also squarefree. However, the conductor of a squarefree monomial ideal need not be squarefree, since for
\[
I=(x_1x_2x_3, x_1x_2x_6, x_1x_3x_5, x_1x_4x_5, x_2x_3x_4, x_2x_3x_5, x_2x_3x_6, x_3x_4x_6)
\]
we have $x_4^2\in C(I)$ but $x_4\notin C(I)$. One thing to note here is that $I$ is not of fiber type while both edge ideals of simple graphs and squarefree Veronese ideals are. This raises the question if conductors of squarfree monomial ideals of fiber type are squarefree, or if there is some additional restriction on $I$ that ensures that $C(I)$ will stay squarefree.
}
\end{Question}

\begin{Question}
{\em
We have seen in \Cref{radical} that $\sqrt{I}\subseteq \sqrt{C(I)}$ if $I$ is an ideal in a Noetherian ring. Further, the stronger relation $I\subseteq C(I)$ is true for both bounded Veronese ideals and edge ideals by \Cref{Graph_conductor} and \Cref{veroneseConductor}. This does not hold in general as for example for the ideal
\[
I=(x_1^2x_5, x_1x_3x_4, x_2^2x_5, x_2x_3^2, x_3^2x_5, x_4^3)
\]
we have $x_1^2x_5\in I$ but $x_1^2x_5\notin C(I)$. However, this ideal is not of fiber type. Computations in \texttt{Macaulay2} suggest that the containment $I\subseteq C(I)$ should hold if $I$ is of fiber type. Can this be proven? 
}
\end{Question}

One case when it is true that $I\subseteq C(I)$ is when $J$ is generated in $y$-degree at most $2$. This happens for example when $I$ has the strong exchange property or is sortable. Indeed, in that case, let $I=(f_1,\ldots, f_m)$ and let $\sum_i c_i y_{j_{2i}}y_{j_{2i+1}}$ for some $c_i\in R$ be a relation. Then using the linear relations $f_ky_i-y_kf_i$, we get
\[
f_k\sum_i c_i y_{j_{2i}}y_{j_{2i+1}} = y_k\sum_i c_if_{j_{2i}}y_{j_{2i+1}} = 0 \pmod{L}.
\]
Thus $f_k\in C(I)$ and hence $I\subseteq C(I)$.


\begin{Question}
{\em 
Comparing \Cref{regularmodI} to \Cref{generalizationvalla}, only one of the results from the former is proven in the latter setting. This raises the question if the second result also holds in that setting. That is, given an ideal $I$ in a Noetherian ring $R$, and an element $f\in R$  such that $I^j:f = I^j$ for all $j\geq 0$, is it true that
\[
C((I,f))\subseteq (C(I),f)?
\]
}
\end{Question}

\end{document}